\documentclass[twocolumn]{autart}    

\usepackage{graphicx}
\graphicspath{{Figures/}}
\usepackage{subcaption}

\usepackage{array, multirow, booktabs}
\usepackage{balance}
\usepackage{cite}
\usepackage{xcolor}
\usepackage{enumitem}

\let\theoremstyle\relax
\usepackage{mathtools, amssymb, bm, amsthm}

\newtheorem{theorem}{Theorem}
\newtheorem{lemma}{Lemma}

\newtheorem{corollary}{Corollary}
\newtheorem{assumption}{Assumption}
\theoremstyle{definition}
\newtheorem{problem}{Problem}

\newtheorem{remark}{Remark}

\allowdisplaybreaks

\DeclareMathOperator*{\T}{\mathsf{T}}

\DeclareMathOperator*{\U}{\mathcal{U}}

\DeclareMathOperator*{\X}{\mathcal{X}}

\newcommand{\vbar}{\overline{v}}
\newcommand{\ubar}{\overline{u}}
\newcommand{\tbar}{\underline{t}}

\begin{document}

\begin{frontmatter}

\title{Ignore Drift, Embrace Simplicity: Constrained Nonlinear Control through Driftless Approximation\thanksref{footnoteinfo}} 

\thanks[footnoteinfo]{This work was supported in part by the Air Force Office of Scientific Research under Grant FA9550-23-1-0131; and in part by the NASA University Leadership Initiative under Grant 80NSSC22M0070. \emph{(Corresponding Author: Ram Padmanabhan.)}}

\author{Ram Padmanabhan}\ead{ramp3@illinois.edu} ~and    
\author{Melkior Ornik}\ead{mornik@illinois.edu}               

\address{University of Illinois Urbana-Champaign, Urbana, IL 61801, USA.}       

\begin{keyword}                           
Nonlinear control, constrained systems, driftless systems.             
\end{keyword}                             

\begin{abstract}                          
We present a novel technique to drive a nonlinear system to reach a target state under input constraints. The proposed controller consists only of piecewise constant inputs, generated from a simple linear driftless approximation to the original nonlinear system. First, we construct this approximation using only the effect of the control input at the initial state. Next, we partition the time horizon into successively shorter intervals and show that optimal controllers for the linear driftless system result in a bounded error from a specified target state in the nonlinear system. We also derive conditions under which the input constraint is guaranteed to be satisfied. On applying the optimal control inputs, we show that the error monotonically converges to zero as the intervals become successively shorter, thus achieving arbitrary closeness to the target state with time. {Using simulation examples on classical nonlinear systems, we illustrate how the presented technique is used to reach a target state while still satisfying input constraints. In particular, we show that our method completes the task even when assumptions of the underlying theory are violated.}
\end{abstract}

\end{frontmatter}

\section{Introduction}
The control of nonlinear dynamical systems presents a number of analytical and computational challenges. Unlike linear systems where, under the assumption of stabilizability \cite{Hespanha}, linear state feedback is guaranteed to achieve asymptotic stability, such a simple controller does not exist for nonlinear systems. These challenges are compounded in scenarios where exact mathematical models for the nonlinear dynamics are unavailable or when practical actuation constraints are taken into account. Feedback linearization and backstepping methods \cite{Khalil_NS} generally require accurate mathematical models, while techniques such as sliding mode control \cite{SMC1, SMC2} can suffer from poor performance and do not directly account for input constraints. More complex methods such as neural network-based control \cite{NN1, NN2} and nonlinear model predictive control \cite{NMPC} require high computational effort.

In this paper, we present a method to control nonlinear systems {with input constraints.} This method is based on a simple linear driftless approximation to the original nonlinear dynamics, constructed using only the effect of the control input at the initial state. By considering a sequence of successively shorter time horizons, we use optimal control laws --- which are easily derived --- for the linear driftless approximation over these horizons. We demonstrate how applying these inputs to the original nonlinear system can ensure any arbitrarily small neighborhood of a target state is achieved. Our work is close in spirit to a number of nonlinear control approaches based on piecewise approximations, and we briefly discuss some of these below.

\subsection{Related Work}
There are a number of methods to drive a nonlinear system to a target state based on piecewise inputs which may use local approximations. Among the earliest efforts in this direction is the work of Oaks and Cook \cite{OC76}, where a Taylor series expansion approximates nonlinear dynamics locally using linear functions. Sontag \cite{Sontag81} presents a number of results in controller and observer design for nonlinear systems using piecewise linear maps. Godhavn \emph{et al.} \cite{GBCS97} describe how a target state can be achieved through a sequence of piecewise constant, bang-bang inputs for a class of nonholonomic systems. Similarly, Luo \emph{et al.} \cite{LT98} derive a recursive algorithm to construct exponentially convergent control laws for a particular class of nonlinear systems. Piecewise nonlinear controllers are presented by Ohtake \emph{et al.} \cite{OTW03} for Takagi-Sugeno fuzzy models constructed from nonlinear dynamics. Piecewise constant controllers and linear models have also been used in the context of nonlinear model predictive control \cite{MPC1, MPC2}. Similar approaches have been proposed to control nonlinear systems without drift, either using piecewise continuous inputs \cite{CS92, GC18} or approximated dynamics \cite{RTM95, MS99}.

While these relatively older approaches influence our proposed method, there also exists recent work in using piecewise inputs or local approximations. Piecewise-affine abstractions of nonlinear stochastic systems are developed in \cite{HWH22}, with similar abstractions used for robust control of nonlinear Markov jump systems in \cite{Zhu22}. Local linear approximations have been used for model predictive control of nonlinear systems \cite{BKMA22}, and piecewise constant feedback control has been used to stabilize time-varying systems \cite{Tan16}. There also exist nonlinear model predictive control algorithms that use linear or piecewise-affine approximations \cite{MPC1, BKMA22, BKMA22b}, and efficient solvers for the same \cite{Casadi, Acados}.

A number of approaches described here do not consider actuation constraints. Among those that do, including \cite{Sontag81, GBCS97, BKMA22, BKMA22b, MPC2}, piecewise inputs are generally designed based on linear model approximations, inducing drift in system dynamics. {Further, the nonlinear MPC approaches described above may still suffer from higher computational burden due to repeatedly solving an optimization problem.} Our approach in this paper uses a linear driftless model instead, resulting in a much simpler approximation {using only the effect of the input} at the initial state. In particular, it is more straightforward to design optimal control signals under input constraints for this model, as shown in \cite{PO25a}.

\subsection{Contributions}
We consider the problem of driving a constrained nonlinear system to a target state over an infinite horizon. {To solve this problem using a simple controller with piecewise constant inputs, we consider a linear driftless approximation of the original nonlinear system using only the effect of the input at the initial state.} We then design a sequence of time instants that diverges to infinity, while the difference between successive terms converges monotonically to zero, thus partitioning the time horizon into successively shorter time intervals. We show that the control signal with minimum energy in each interval for the linear driftless approximation leads to a bounded error from the target state in the original nonlinear system. Simultaneously, we develop conditions on a parameter governing the sequence which guarantee the input constraint is satisfied. Finally, we show that as the intervals become monotonically shorter, the error bound monotonically decreases to zero, thus reaching any arbitrarily small neighborhood of the target state. The use of this method is illustrated through a set of simulation examples on common nonlinear systems. Additionally, we show that satisfactory performance may be achieved even when some underlying assumptions are violated.

The remainder of this paper is organized as follows. In Section \ref{sec:Formulation}, we formulate the problem statement with basic assumptions, and present the notion of a linear driftless approximation. In Section \ref{sec:Method}, we describe our technique, including the partition of the time horizon and corresponding controllers in each interval. Our primary convergence results are presented in Section \ref{sec:Properties}, where we construct conditions which ensure input constraints are satisfied, and prove that the target state can be reached with arbitrary precision. Section~\ref{sec:Examples} presents a set of examples illustrating the technique. In Section \ref{sec:Discussion}, we provide additional remarks on the merits and drawbacks of this technique, as well as a number of avenues for future work.

\section{Problem Formulation} \label{sec:Formulation}
We consider {control-affine nonlinear dynamical systems evolving on a state space $\X \subseteq \mathbb{R}^N$}:
\begin{equation} \label{eq:System}
	{\dot{x}(t) = f(x(t)) + g(x(t))u(t), ~x(0) = x_0 \in \X}
\end{equation}
where $x(t) \in \X$ is the state and $u(t) \in \mathbb{R}^m$ is the control input, $f:\X \to \mathbb{R}^N$ is the function containing drift terms and $g:\X \to \mathbb{R}^{N\times m}$ is the input function. {We assume $f$ and $g$ are $D_f$- and $D_g$-Lipschitz continuous in the $\infty$-norm on the state space $\X$. In other words, for any $x_1, x_2 \in \X$, there exist constants $D_f$ and $D_g$ such that}
\begin{align}
f(x_1) - f(x_2) &= d_f(x_1, x_2), \nonumber \\
\text{where } \|d_f(x_1, x_2)\|_{\infty} &\leq D_f \|x_1 - x_2\|_{\infty}, \label{eq:f_cond} \\
\text{and }~~ g(x_1) - g(x_2) &= d_g(x_1, x_2), \nonumber \\
\text{where } \|d_g(x_1, x_2)\|_{\infty} &\leq D_g \|x_1 - x_2\|_{\infty}, \label{eq:g_cond}
\end{align}
representing growth rate constraints on the system dynamics. {We remark that by the finite subcover property of compact sets, local Lipschitz continuity is equivalent to Lipschitz continuity on $\X$ if $\X$ is compact \cite{Lipschitz}.} We also assume the input $u$ is constrained as
\begin{align}
u \in \U \coloneqq \bigl\{&u: [0, \infty) \to \mathbb{R}^{m}: u \text{ is piecewise continuous} \nonumber \\
&\text{in $t$, } \|u(t)\|_{\infty} \leq 1 \text{ for all $t$}\bigr\}, \label{eq:SetU}
\end{align}
representing practical actuator constraints in line with prior work \cite{KB60, PO25a}. The central problem we address in this paper is stated below. \vspace{0.25cm}

\begin{problem} \label{prob}
Given an initial state $x_0 \in \mathbb{R}^N$ and target state $x_{tg} \in \mathbb{R}^N$, design a control input $u \in \U$ which ensures that for any $\varepsilon > 0$, there exists a time $\tbar$ such that $\|x_{tg} - x(t)\|_{\infty} < \varepsilon$ for all $t \geq \tbar$.
\end{problem}
We remark that a systematic procedure to solve such a problem does not exist for nonlinear systems, unlike linear systems \cite{Hespanha}.
The approach we take in this paper uses a linear driftless approximation to system \eqref{eq:System}. Let $g_0 = g(x_0)$. Then, the dynamics
\begin{equation} \label{eq:Driftless}
	\dot{x}(t) = g_0u(t), ~x(0) = x_0,
\end{equation}
are termed the \emph{linear driftless approximation} to \eqref{eq:System}. We now make the following assumption for this approximation. \vspace{0.25cm}

\begin{assumption} \label{asm:ctrb}
The matrix $g_0$ has full row rank $N$.
\end{assumption}

Under this assumption, any target state $x_{tg}$ can be achieved starting from any initial state $x_0$ in the linear driftless system \eqref{eq:Driftless}, {thus guaranteeing controllability of \eqref{eq:Driftless}.} We note that Assumption \ref{asm:ctrb} is violated for systems with more states than inputs; however, we briefly discuss in Section~\ref{sec:Properties} and illustrate in Section \ref{sec:Examples} how satisfactory performance can still be achieved if this assumption is not satisfied. {Similarly, while \eqref{eq:f_cond} and \eqref{eq:g_cond} are strong requirements on the dynamics \eqref{eq:System}, we illustrate in Section \ref{sec:Examples} that Problem \ref{prob} can still be solved for systems violating these conditions.}

It is also worth noting here that the dynamics \eqref{eq:System} along with the constraint \eqref{eq:SetU} can be rewritten as $\dot{x} = \hat{u}$, $\hat{u} \in \hat{\U}(x)$ where $\hat{\U}(x) \coloneqq \{\hat{u}:\hat{u} = f(x) + g(x)u, ~u \in \U\}$ for each $x$. In other words, we have a simple integrator with state-dependent input constraints that are not necessarily convex. Controlling such systems is known to be difficult \cite{LCF22, CVOC} due to the complexity of these constraints.

In the following section, we describe our method to solve Problem \ref{prob} using controllers solving this problem for \eqref{eq:Driftless}.

\section{Method} \label{sec:Method}
In this section, we describe our solution to Problem \ref{prob}, which successively uses controllers based on the linear driftless dynamics \eqref{eq:Driftless}. We show that applying a controller with minimum energy for the system \eqref{eq:Driftless} leads to {bounded error from the target state} in original nonlinear system \eqref{eq:System}. This principle is used to apply such controllers over successively shorter time intervals, and is shown to eventually lead to asymptotic convergence of the state $x(t)$ to the target $x_{tg}$. {It is worth remarking that in comparison to linearization-based methods, the approximation \eqref{eq:Driftless} uses significantly less information about the dynamics \eqref{eq:System}. In particular, the effect of the drift term $f(x)$ is ignored. Nevertheless, by exploiting the Lipschitz continuity of $f(x)$ and $g(x)$ and recomputing the controller in each time interval, we are able to solve Problem~\ref{prob}.}

Consider the system \eqref{eq:Driftless} and pick a variable $t^* \in (0, \infty)$, behaving as a design parameter. Our first step is to design a controller which ensures $x(t^*) = x_{tg}$ in \eqref{eq:Driftless}. Writing the solution to this system at $t^*$, and setting $x(t^*) = x_{tg}$,
\begin{equation} \label{eq:DftSol}
	x(t^*) - x_0 = t^* g_0 \ubar = x_{tg}-x_0,
\end{equation}
where $\ubar = \frac{1}{t^*}\int_{0}^{t^*}u(t)\mathrm{d}t$ and we aim to design $u(t)$ to achieve the second part of the equality above. Under Assumption \ref{asm:ctrb}, it is easy to show that the constant control input
\begin{equation} \label{eq:u1}
	u_1(t) \equiv \frac{1}{t^*}g_{0}^{\dagger}(x_{tg}-x_0) \text{ for all $t \in [0, t^*]$}
\end{equation}
using the pseudoinverse $g_{0}^{\dagger}$ of $g_0$ will achieve $x(t^*) = x_{tg}$ in \eqref{eq:Driftless}. Indeed, it can be shown \cite{PO25a} that \eqref{eq:u1} is the control input with minimum energy that achieves $x(t^*) = x_{tg}$ in \eqref{eq:Driftless}.

We now investigate the response of the original nonlinear system \eqref{eq:System} to this constant control input. Substituting \eqref{eq:u1} in \eqref{eq:System}, using the property \eqref{eq:g_cond} and noting that $g_0 \int_{0}^{t^*}u_1(t)\mathrm{d}t = x_{tg}-x_0$,
$$
	x(t^*)-x_0 = x_{tg}-x_0 + \int_{0}^{t^*} (f(x(t)) + d_g(x(t), x_0)u_1(t))\mathrm{d}t
$$
so that
\begin{align}
	&x_{tg}-x(t^*) \nonumber \\
	&~= v(0,t^*) \coloneqq -\int_{0}^{t^*} \left(f(x(t)) + d_g(x(t), x_0)u_1(t)\right) \mathrm{d}t. \label{eq:NLSol1}
\end{align}
Thus, on applying the input \eqref{eq:u1} to the nonlinear system \eqref{eq:System}, the system state is guaranteed to be within $v(t^*)$ from the target state. We now exploit this fact to design a sequence of constant inputs of the form \eqref{eq:u1} that solves Problem \ref{prob}. 

\begin{figure}[!t]
	\centering
	\includegraphics[width = 0.47\textwidth]{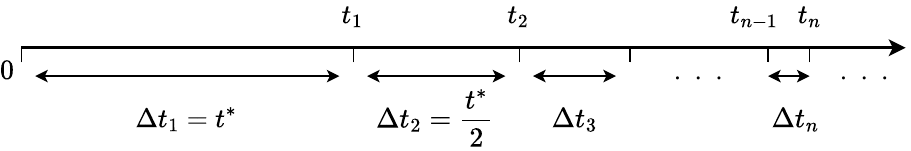}
	\caption{Partitioning the time horizon using the constructed sequence of time instants $\{t_n\}$.}
	\label{fig:Seq}
\end{figure}

Construct the infinite sequence of time instants $\{t_n\}$ defined by $t_0 = 0$ and $t_n = \sum_{k = 1}^{n} \frac{t^*}{k}$ for $n \geq 1$ as shown in Fig.~\ref{fig:Seq}. This sequence is the set of partial sums for the harmonic series scaled by $t^*$, the design variable. Further, let $\Delta t_n = t_{n}-t_{n-1} = \frac{t^*}{n}$. Clearly, $\lim_{n \to \infty} t_n = \infty$ and $\lim_{n \to \infty} \Delta t_n = 0$. Let $x_1 = x(t_1) = x(t^*)$ denote the state achieved on applying $u_1(t)$ in \eqref{eq:u1} over $t \in [0, t^*]$ to the system \eqref{eq:System}, so that \eqref{eq:NLSol1} reduces to
\begin{equation} \label{eq:NLSol11}
x_{tg} - x_1 = v(0,t_1).
\end{equation}
The key idea in our approach is to redesign a controller of the form \eqref{eq:u1} at each instant $t_n$ based on the state achieved at that instant. Thus, for $n \geq 2$, at $t_{n-1}$, we apply the controller
\begin{equation} \label{eq:un}
	u_n(t) \equiv \frac{1}{\Delta t_n} g_{0}^{\dagger} (x_{tg} - x_{n-1}) \text{ for all $t \in [t_{n-1}, t_n]$}.
\end{equation}
Substituting \eqref{eq:un} in \eqref{eq:System} over the horizon $[t_{n-1}, t_n]$, using the property \eqref{eq:g_cond} and noting that $g_0 \int_{t_{n-1}}^{t_n} u_n(t) \mathrm{d}t = x_{tg} - x_{n-1}$, 
\begin{align}
	x_{tg} - x_n &= v(t_{n-1},t_n) \nonumber \\
	&\coloneqq -\int_{t_{n-1}}^{t_n} \left(f(x(t)) + d_g(x(t), x_0)u_n(t)\right) \mathrm{d}t \label{eq:NLSoln}
\end{align}
where $x_n \coloneqq x(t_n)$ denotes the state achieved at $t_n$. Here, we use the solution of \eqref{eq:System} over the horizon $[t_{n-1}, t_n]$, and are guaranteed to be within $v(\Delta t_n)$ from $x_{tg}$ at $t_n$. This procedure is repeated successively, and in general for $n \geq 1$, the control \eqref{eq:un} guarantees \eqref{eq:NLSoln} in the nonlinear system \eqref{eq:System}.

%

We note here that the sequence of inputs \eqref{eq:un} uses only the linear driftless approximation \eqref{eq:Driftless}, i.e., solely the knowledge of the input function at the initial state. In the next section, we investigate the convergence properties of this method, proving that the sequence of inputs \eqref{eq:un} solves Problem \ref{prob}. We also develop conditions on $t^*$ that guarantee the input constraint \eqref{eq:SetU} is satisfied.

\section{Convergence and Boundedness Properties} \label{sec:Properties}
In this section, we prove that the system state converges asymptotically to the target state under the method described in Section \ref{sec:Method}. We first develop conditions on $t^*$ which guarantee that the constraint \eqref{eq:SetU} is satisfied, simultaneously showing that the error from $x_{tg}$ at each instant $t_n$ is bounded. Next, we exploit the fact that the error converges to zero with time, thus showing that Problem \ref{prob} is solved using the sequence of inputs \eqref{eq:un}. First, we present two lemmas that are used to prove later results in this section. \vspace{0.25cm}

\begin{lemma} \label{lem:vn}
Let $D_S = D_f + D_g$ and {$c \coloneqq \left\|f(x_0)\right\|_{\infty} + D_S\left\|x_{tg}-x_0\right\|_{\infty}$.} Consider the sequence
\begin{equation} \label{eq:vn}
\vbar_n \coloneqq \frac{c}{D_S}\left(e^{\Delta t_n D_S} - 1\right).
\end{equation}
Then, for all $n \geq 1$,
\begin{equation} \label{eq:vn_rec}
	\vbar_{n+1} \leq \frac{n}{n+1}\vbar_n,
\end{equation}
and hence $\vbar_n \leq \frac{1}{n}\vbar_1$ and $\lim_{n \to \infty} \vbar_n = 0$.
\end{lemma}

\begin{proof}
Note that $\Delta t_{n+1} = \frac{n}{n+1}\Delta t_n$. Using a Taylor series expansion, {we have $e^{\Delta t_{n+1}D_S} = 1 + \Delta t_{n+1}D_S + \frac{(\Delta t_{n+1}D_S)^2}{2!} + \ldots$, so that $e^{\Delta t_{n+1}D_S} - 1 = \sum_{k = 1}^{\infty} \frac{(\Delta t_{n+1}D_S)^k}{k!}$. Then,
}
\begin{align*}
&e^{\Delta t_{n+1} D_S}-1 = \sum_{k = 1}^{\infty} \left(\frac{n}{n+1}\right)^k \frac{(\Delta t_nD_S)^k}{k!} \\
&~\leq \frac{n}{n+1} \sum_{k = 1}^{\infty} \frac{(\Delta t_nD_S)^k}{k!} = \frac{n}{n+1} \left(e^{\Delta t_n D_S}-1\right),
\end{align*}
since $\left(\frac{n}{n+1}\right)^k \leq \frac{n}{n+1}$ for all $k \geq 1$, and each term in the summation is positive. Using \eqref{eq:vn}, result \eqref{eq:vn_rec} follows. Then, $\vbar_n \leq \frac{1}{n}\vbar_1$ and thus $\lim_{n \to \infty} \vbar_n = 0$.
\end{proof}

\begin{lemma} \label{lem:c1}
{The equation $e^y = \frac{y}{c_1} + 1$ has a real, positive solution for $y$ if and only if $0 < c_1 < 1$.}
\end{lemma}

\begin{proof}
{The proof follows from \cite{FM05}, where it is shown that the function $\frac{y}{e^y-1}$ maps positive values of $y$ uniquely to the interval $(0,1)$. Since the given equation can be written as $\frac{y}{e^y-1} = c_1$, it follows that a real, positive solution for $y$ exists if and only if $0 < c_1 < 1$, and this solution is unique.}

\end{proof}

We now present the main result of this paper, developing conditions on $t^*$ to ensure \eqref{eq:SetU} is satisfied and simultaneously providing an explicit bound on the quantity in \eqref{eq:NLSoln}.  \vspace{0.25cm}

\begin{theorem}[Input Constraint and Boundedness] \label{thm:bound}
Assume $c_1 \coloneqq 2c\left\|g_{0}^{\dagger}\right\|_{\infty} < 1$, where $c$ is defined in Lemma \ref{lem:vn}. {Then, input constraint \eqref{eq:SetU} is satisfied by the sequence of inputs \eqref{eq:un} if $t^*$ satisfies
\begin{equation} \label{eq:t_cond}
\left\|g_{0}^{\dagger}(x_{tg}-x_0)\right\|_{\infty} \leq t^* \leq \overline{t},
\end{equation}
where $\overline{t}$ is the only real positive solution of the transcendental equation $e^{tD_S} = \frac{tD_S}{c_1} + 1$ for $t$. Under this condition,} for all $n \geq 1$, $x_{tg} - x_n$ is bounded by
\begin{equation} \label{eq:vn_bound}
\|x_{tg} - x_n\|_{\infty} \leq \vbar_n = \frac{c}{D_S}\left(e^{\Delta t_n D_S} - 1\right).
\end{equation}
\end{theorem}

\begin{proof}


The proof proceeds by induction, and is split into two parts. In the first part, we prove the claim for $n = 1$. In the second part, we show that if the claim holds for some $n = k$, $k \geq 1$, then it must hold for $n = k+1$.

\underline{Part $1$: Proof for $n = 1$}: Consider $u_1(t)$ in \eqref{eq:u1}. {Clearly, if the condition
\begin{equation} \label{eq:t_cond1}
t^* \geq \left\|g_{0}^{\dagger}(x_{tg}-x_0)\right\|_{\infty},
\end{equation}
is satisfied, then $\|u_1(t)\|_{\infty} \leq 1$ for all $t \in [0,t^*]$, i.e., $u_1 \in \U$. This proves one part of condition \eqref{eq:t_cond}.} To prove the bound \eqref{eq:vn_bound} for $n = 1$, consider the quantity
\begin{equation} \label{eq:vt}
v(0,t) = -\int_{0}^{t} \left(f(x(\tau)) + d_g(x(\tau), x_0)u_1(\tau)\right) \mathrm{d}\tau
\end{equation}
for $t \in [0, t^*]$, based on \eqref{eq:NLSol1}. Using condition \eqref{eq:f_cond}, $f(x(\tau)) = f(x_0) + d_f(x(\tau), x_0)$. Substituting in \eqref{eq:vt},
\begin{align} 
v(0,t) = -\Bigl[tf(x_0) &+ \int_{0}^{t} d_f(x(\tau), x_0)\mathrm{d}\tau \nonumber \\
&+ \int_{0}^{t} d_g(x(\tau), x_0)u_1(\tau) \mathrm{d}\tau\Bigr]. \label{eq:prop1}
\end{align}
We now apply $\|\cdot\|_{\infty}$ on both sides of \eqref{eq:prop1} and use Jensen's inequality \cite{NPBook} to move the norm inside the integral. Then, using properties \eqref{eq:f_cond} and \eqref{eq:g_cond} and noting that \eqref{eq:t_cond1} guarantees $\|u_1(\tau)\|_{\infty} \leq 1$ for all $\tau \in [0, t^*]$,
\begin{align}
\|v(0,t)\|_{\infty} &\leq t \left\|f(x_0)\right\|_{\infty} + D_S \int_{0}^{t} \|x(\tau) - x_0\|_{\infty} \mathrm{d}\tau \nonumber \\
&\leq t \left(\left\|f(x_0)\right\|_{\infty} + D_S \|x_{tg} - x_0\|_{\infty}\right) \nonumber \\
&+ D_S \int_{0}^{t} \|x_{tg} - x(\tau)\|_{\infty} \mathrm{d}\tau, \label{eq:prop2}
\end{align}
where $D_S \coloneqq D_f + D_g$ and we use the triangle inequality $\|x(\tau) - x_0\|_{\infty} \leq \|x_{tg} - x_0\|_{\infty} + \|x_{tg} - x(\tau)\|_{\infty}$. Let $Q^1(t) \coloneqq \int_{0}^{t} \|x_{tg} - x(\tau)\|_{\infty}\mathrm{d}\tau$ for $t \in [0, t^*]$, with $Q^1(0) = 0$. From the first fundamental theorem of calculus \cite{Apostol}, $\dot{Q}^1(t) = \|x_{tg} - x(t)\|_{\infty}$, and we are interested in bounding $\dot{Q}^1(t^*) = \|x_{tg} - x(t^*)\|_{\infty} = \|x_{tg} - x_1\|_{\infty}$. For $t \in [0, t^*]$, consider the input $u_1(t)$ in \eqref{eq:u1} applied over the horizon $[0, t]$. Note that $u_1(t)$ is designed such that $g_0u_1(t) = \frac{1}{\Delta t_1}(x_{tg} - x_0)$ as $\Delta t_1 = t_1 = t^*$. Then, writing out the solution to \eqref{eq:System} and using \eqref{eq:g_cond},
\begin{align*}
&x(t) - x_0 \\
&~ = \int_{0}^{t}f(x(\tau))\mathrm{d}\tau + \int_{0}^{t} d_g(x(\tau), x_0)u_1(\tau)\mathrm{d}\tau \\
&~+ g_0\int_{0}^{t} u_1(\tau)\mathrm{d}\tau = -v(0,t) + \frac{t}{\Delta t_1} (x_{tg} - x_0).
\end{align*}
Rearranging,
$$
x_{tg} - x(t) = v(0,t) + \frac{t_1-t}{\Delta t_1} (x_{tg} - x_0).
$$
{Taking norms on both sides,
\begin{align}
\dot{Q}^1(t) &= \|x_{tg} - x(t)\|_{\infty} \nonumber \\
&\leq \|v(0,t)\|_{\infty} + \frac{t_1-t}{\Delta t_1}\|x_{tg} - x_0\|_{\infty} \nonumber \\
&\leq \|v(0,t)\|_{\infty} + \|x_{tg} - x_0\|_{\infty} \text{ for all } t \in [0, t^*],\label{eq:normbound1}
\end{align}
since $\frac{t_1-t}{\Delta t_1} \leq 1$. Adding $\|x_{tg} - x_0\|_{\infty}$ to both sides of \eqref{eq:prop2},} using the definition of $Q^1(t)$ and \eqref{eq:normbound1}, \eqref{eq:prop2} can be rewritten as the differential inequality
\begin{align}
	&{\dot{Q}^1(t) \leq ct + D_S Q^1(t) + \|x_{tg} - x_0\|_{\infty}}, ~t \in [0, t^*], \label{eq:propQ}
\end{align}
with $Q^1(0) = 0$ and where $c$ is defined in Lemma \ref{lem:vn}. Let $P^1(t) \coloneqq ct + D_SQ^1(t) + \|x_{tg} - x_0\|_{\infty}$ so that $P^1(0) = 0$ and $\dot{Q}^1(t) \leq P^1(t)$. Thus, $\dot{P}^1(t) = c + D_S\dot{Q}^1(t) \leq c + D_SP^1(t)$. Subsequently, let $R^1(t) \coloneqq c + D_SP^1(t)$ so that $R^1(0) = c$ and $\dot{P}^1(t) \leq R^1(t)$. Then, 
\begin{equation} \label{eq:propQ2}
\dot{R}^1(t) = D_S\dot{P}^1(t) \leq D_S R^1(t), ~t \in [0, t^*], ~R^1(0) = c,
\end{equation}
and we have a linear differential inequality in $R^1(t)$. Applying Gr{\"o}nwall's inequality \cite{BGP} on the horizon $[0, t^*]$ in \eqref{eq:propQ2}, we have $R^1(t) \leq ce^{t D_S}$ so that
\begin{equation} \label{eq:propQ1}
	P^1(t) = \frac{1}{D_S}(R^1(t) - c) \leq \frac{c}{D_S}\left(e^{t D_S}-1\right)
\end{equation}
from the definition of $R^1(t)$. Finally, we substitute $t = t^* = \Delta t_1$, and use \eqref{eq:prop2}, \eqref{eq:propQ}, \eqref{eq:propQ1} and the fact that $\dot{Q}^1(t) \leq P^1(t)$ for $t \in [0, t^*]$ to obtain
\begin{align}
\|x_{tg} &- x_1\|_{\infty} = \dot{Q}^1(t^*) \nonumber \\
&\leq P^1(t^*) \leq \vbar_1 \coloneqq \frac{c}{D_S}\left(e^{\Delta t_1 D_S} - 1\right), \label{eq:v1}
\end{align}
proving the bound \eqref{eq:vn_bound} for $n = 1$.

\underline{Part 2: Proof for $n \geq 1$}: We now make the induction hypotheses that for some $n = k$ and $k \geq 1$, $\|u_k(t)\|_{\infty} \leq 1$ for all $t \in [t_{k-1}, t_k]$ under condition \eqref{eq:t_cond}, and that $\|x_{tg} - x_k\|_{\infty} \leq \vbar_k$. Then, based on \eqref{eq:un}, consider
\begin{equation} \label{eq:uk1}
	u_{k+1}(t) \equiv \frac{1}{\Delta t_{k+1}} g_{0}^{\dagger} (x_{tg} - x_k) \text{ for all } t \in [t_k, t_{k+1}].
\end{equation}
{We now show that $u_{k+1} \in \U$ if condition \eqref{eq:t_cond} is satisfied. Rewriting $\|u_{k+1}(t)\|_{\infty} \leq 1$, $\left\|\frac{1}{\Delta t_{k+1}} g_{0}^{\dagger} (x_{tg} - x_k)\right\|_{\infty} \leq 1$, which is satisfied if
\begin{equation} \label{eq:t_cond23}
\frac{k+1}{t^*} \left\|g_{0}^{\dagger}\right\|_{\infty} \vbar_k \leq 1,
\end{equation}
using the submultiplicative property of norms, the inductive hypothesis $\|x_{tg} - x_k\|_{\infty} \leq \vbar_k$, and using $\Delta t_{k+1} = \frac{t^*}{k+1}$. Next, we note $\vbar_k \leq \frac{1}{k}\vbar_1$ from Lemma~\ref{lem:vn}, and that $\frac{k+1}{k} \leq 2$ for all $k \geq 1$. Thus, \eqref{eq:t_cond23} is satisfied if
\begin{equation} \label{eq:t_cond21}
\frac{2}{t^*}\left\|g_{0}^{\dagger}\right\|_{\infty} \vbar_1 \leq 1
\end{equation}
is satisfied. Rearranging, 
$$
t^* \geq 2\left\|g_{0}^{\dagger}\right\|_{\infty} \vbar_1 = 2\left\|g_{0}^{\dagger}\right\|_{\infty} \frac{c}{D_S} \left(e^{t^*D_S}-1\right),
$$
or, substituting $c_1 = 2\left\|g_{0}^{\dagger}\right\|_{\infty}c$,
\begin{equation} \label{eq:t_cond22}
	e^{t^*D_S} \leq \frac{t^*D_S}{c_1} + 1.
\end{equation}
Consider the function $h(t) = e^{tD_S}-\frac{tD_S}{c_1}-1$ for $t \geq 0$. Then $\frac{dh(t)}{dt} = D_S\left(e^{tD_S}-\frac{1}{c_1}\right)$, which increases from negative to arbitrarily large positive values as $t$ increases, since $c_1 < 1$ as assumed in the theorem. Thus, for $t \geq 0$, $h(t)$ starts at $h(0) = 0$, initially decreases and then increases without bound. The only other zero crossing is given by $h\left(\overline{t}\right) = 0$, where $\overline{t}$ is a unique real positive number from Lemma~\ref{lem:c1}. Thus, \eqref{eq:t_cond22} is satisfied when
\begin{equation} \label{eq:t_cond2}
	h(t) \leq 0, \text{ or } t^* \leq \overline{t},
\end{equation}
or $u_{k+1} \in \U$ if $t^* \leq \overline{t}$, proving the second part of \eqref{eq:t_cond}.
}

To prove \eqref{eq:vn_bound} for $n = k+1$, consider
\begin{equation} \label{eq:vtk}
v(t_k,t) = -\int_{t_k}^{t} \left(f(x(\tau)) + d_g(x(\tau), x_0)u_{k+1}(\tau)\right) \mathrm{d}\tau
\end{equation}
for $t \in [t_k, t_{k+1}]$, based on \eqref{eq:NLSoln}. Using arguments as used in Part $1$, we can show that
\begin{align*}
\|v(t_k, t)\|_{\infty} &\leq (t-t_k)\left(\left\|f(x_0)\right\|_{\infty} + D_S \|x_{tg} - x_0\|_{\infty}\right) \\
&+ D_S \int_{t_k}^{t} \|x_{tg} - x(\tau)\|_{\infty} \mathrm{d}\tau
\end{align*}
similar to \eqref{eq:prop2}, since $\|u_{k+1}(t)\|_{\infty} \leq 1$ as proved above. Then, define $Q^{k+1}(t) \coloneqq \int_{t_k}^{t} \|x_{tg} - x(\tau)\|_{\infty}\mathrm{d}\tau$ for $t \in [t_k, t_{k+1}]$ with $Q^{k+1}(t_k) = 0$. Using arguments identical to Part $1$, it can be shown that
\begin{equation} \label{eq:Qk1dot}
\dot{Q}^{k+1}(t) = \|x_{tg} - x(t)\|_{\infty} \leq \frac{c}{D_S}\left(e^{(t-t_k)D_S}-1\right),
\end{equation}
and thus
\begin{align}
\|x_{tg} - x(t_{k+1})\|_{\infty} &= \|x_{tg} - x_{k+1}\|_{\infty} \nonumber \\
&\leq \vbar_{k+1} = \frac{c}{D_S}\left(e^{\Delta t_{k+1}D_S}-1\right) \label{eq:vk1}
\end{align}
{as defined in Lemma \ref{lem:vn},} thus proving the bound \eqref{eq:vn_bound} for $n = k+1$. Hence, using induction, the input constraint \eqref{eq:SetU} is satisfied by the sequence of inputs \eqref{eq:un}, and the bound \eqref{eq:vn_bound} holds for all $n$, concluding the proof.

\end{proof}

We now exploit Lemma \ref{lem:vn} and Theorem \ref{thm:bound} to prove that the sequence of inputs \eqref{eq:un} solves Problem \ref{prob}. \vspace{0.25cm}

\begin{corollary}[Convergence to Target State] \label{cor:conv}

\phantom{}
\begin{enumerate}[label = $(\roman{enumi})$]
	\item For $n \geq 1$, the sequence of control inputs \eqref{eq:un} ensures that for any $\varepsilon > 0$, there exists some $\underline{n}$ such that $\|x_{tg} - x_n\|_{\infty} < \varepsilon$ for all $n \geq \underline{n}$.
	\item For any $\varepsilon > 0$, there exists some $\tbar$ such that $\|x_{tg} - x(t)\|_{\infty} < \varepsilon$ for all $t \geq \tbar$, thus solving Problem \ref{prob}.
\end{enumerate}

\end{corollary}

\begin{proof}
~
\phantom{}
~
\begin{enumerate}[label = $(\roman{enumi})$]

\item From \eqref{eq:vn_rec}, there clearly exists some $\underline{n}$ such that $\vbar_n < \varepsilon$ for all $n \geq \underline{n}$, for any $\varepsilon > 0$. Using \eqref{eq:vn_bound}, the result follows.

\item It is easy to see that part $(i)$ of this proof solves a weaker version of Problem \ref{prob}, where the guarantee holds only at the time instants $t_n$ for $n \geq \underline{n}$. From \eqref{eq:Qk1dot}, we also know that for $t_n < t < t_{n+1}$, $\|x_{tg} - x(t)\|_{\infty} \leq \frac{c}{D_S}\left(e^{(t-t_n)D_S}-1\right) \leq \frac{c}{D_S}\left(e^{\Delta t_{n+1}D_S}-1\right) = \vbar_{n+1}$. From part $(i)$, we know that for $n \geq \underline{n}$, $\vbar_{n+1} < \varepsilon$ for all $\varepsilon > 0$. Then, we let $\tbar = t_{\underline{n}}$ and note that any $t > \tbar$ is in some $[t_n, t_{n+1}]$ for $n \geq \underline{n}$. From the above arguments, for all $t \geq \tbar$, $\|x_{tg} - x(t)\|_{\infty} < \varepsilon$ for all $\varepsilon > 0$.



\end{enumerate}

\end{proof}

Corollary \ref{cor:conv} is the culmination of the proof that the method in Section \ref{sec:Method} solves Problem \ref{prob}, and is based on the groundwork laid in Lemma \ref{lem:vn} and Theorem \ref{thm:bound}. A simple illustration of the convergence behavior is provided in Fig.~\ref{fig:Traj}, showing that the state $x_n$ at each instant $t_n$ is within $\vbar_n$ of the target state $x_{tg}$. \vspace{0.25cm}

\begin{figure}[!t]
	\centering
	\includegraphics[width = 0.3\textwidth]{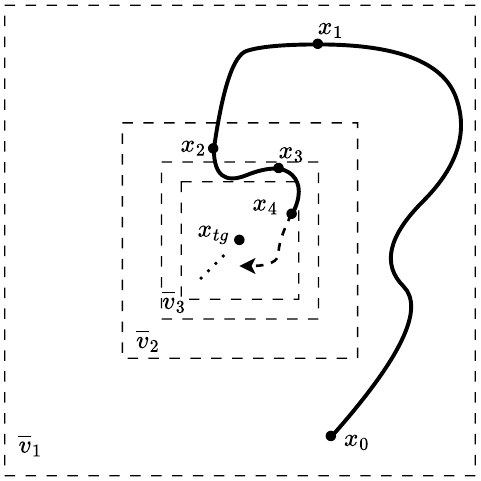}
	\caption{A simple illustration of convergence behavior from Theorem \ref{thm:bound} and Corollary \ref{cor:conv}. Here, $x_0$ is the initial state, $x_{tg}$ is the target state, and $x_n = x(t_n)$ denotes the state achieved at each instant $t_n$. The dashed boxes denote the $\infty$-norm balls with radius $\vbar_n$ achieved according to \eqref{eq:vn_bound}. While $x_n$ is guaranteed to be within $\vbar_n$ of $x_{tg}$, in practice, $x_n$ may be even closer to $x_{tg}$ and may lie within $\vbar_k$ of $x_{tg}$ for $k > n$.}
	\label{fig:Traj}
\end{figure}

\begin{remark}
{While \eqref{eq:t_cond} guarantees that the input constraint \eqref{eq:SetU} is satisfied, we note that condition \eqref{eq:t_cond} is sufficient and not necessary.} In particular, the right-hand side $t^* \leq \overline{t}$ can be extremely constraining, since the bound in \eqref{eq:t_cond23} can be very conservative. We show in the next section that $t^*$ can simply be viewed as a tuning parameter, and even if there does not exist a $t^*$ satisfying \eqref{eq:t_cond}, there exist \emph{reasonable} choices of $t^*$ that will often lead to the target state being achieved and and input constraint being satisfied.
\end{remark}

\vspace{0.25cm}

\begin{remark}[Relaxing Assumption \ref{asm:ctrb}] \label{rem:relax}
{We briefly address the case when $g_0$ does not have full row rank, i.e., Assumption~\ref{asm:ctrb} is violated. At $x(t_{n-1}) = x_{n-1}$, we apply the control input $u_n$ in \eqref{eq:un} over the interval $[t_{n-1}, t_n]$. In the driftless system \eqref{eq:Driftless}, this input will not necessarily achieve $x(t_n) = x_{tg}$ as Assumption~\ref{asm:ctrb} is violated. Let $x_{n}^{dft}$ denote the state achieved by \eqref{eq:Driftless} at $t_n$, so that
$$
    g_0 \int_{t_{n-1}}^{t_n} u_n(t)\mathrm{d}t = \Delta t_n g_0 \ubar_n = x_{n}^{dft} - x_{n-1}.
$$
The quantity $x_{n}^{dft}$ is such that 
\begin{align*}
    x_{tg} - x_{n}^{dft} &= (x_{tg} - x_{n-1}) - (x_{n}^{dft} - x_{n-1}) \\
    &= (I - g_0g_{0}^{\dagger})(x_{tg} - x_{n-1}),
\end{align*}
substituting $u_n(t)$ from \eqref{eq:un}. The above expression quantifies the error from the target state in the linear driftless dynamics. In the original nonlinear dynamics \eqref{eq:System}, applying \eqref{eq:un} over the interval $[t_{n-1}, t_n]$, 
\begin{align*}
&x_n - x_{n-1} = \int_{t_{n-1}}^{t_n} f(x(t))\mathrm{d}t + g_0 \int_{t_{n-1}}^{t_n} u_n(t) \mathrm{d}t \\
&\! + \!\int_{t_{n-1}}^{t_n} \! \!d_g(x(t), x_0)u_n(t)\mathrm{d}t = -v(t_{n-1}, t_n) + x_{n}^{dft} - x_{n-1}.
\end{align*}
Subtracting $x_{tg}$ on both sides and rearranging, the error from the target state in the nonlinear system at $t_n$ is then
$$
x_{tg} - x_n = v(t_{n-1}, t_n) + (I - g_0g_{0}^{\dagger})(x_{tg} - x_{n-1}).
$$
Comparing with \eqref{eq:NLSoln}, we note that the term $(I - g_0g_{0}^{\dagger})(x_{tg} - x_{n-1})$ adds to $v(t_{n-1}, t_n)$. Since this relation holds for all $n$, we also note that $x_{tg} - x_{n-1} = v(t_{n-2}, t_{n-1}) + (I - g_0g_{0}^{\dagger}) (x_{tg} - x_{n-2})$, and the error from the target state is thus impacted by the cumulative effect of $(I - g_{0}g_{0}^{\dagger})^n$. Our examples in the next section illustrate that these additional terms can decay to zero for some $g_0$, as the intervals of time become shorter (i.e., $\Delta t_n \to 0$). Thus, the target state may be achieved despite Assumption 1 being violated.
}
\end{remark}

\section{Examples} \label{sec:Examples}
We now present a set of case studies that show how Problem~\ref{prob} can be solved for a variety of systems using the procedure described in Section~\ref{sec:Method}. Our first example satisfies the basic assumptions described in Section~\ref{sec:Formulation}, i.e., conditions \eqref{eq:f_cond}, \eqref{eq:g_cond} and Assumption \ref{asm:ctrb}. However, in subsequent examples, we relax these assumptions and show in simulation that satisfactory performance can still be achieved using the sequence of control inputs \eqref{eq:un}. The code for our method is available online\footnote{https://github.com/leadcatlab/Constrained-NL-Control}.

\subsection{Example 1: ADMIRE fighter jet model with wind effects}
We consider the dynamics of the ADMIRE fighter jet model subjected to wind effects as discussed in \cite{PO25b}, based on data from \cite{SDRA} and wind models in \cite{WWS24}. These dynamics can be written as $\dot{x} = Ax+Bu+f_w(x)$, where
\begin{align}
\begin{split} \label{eq:ADMIRE}
x &= \begin{bmatrix} p \\ q \\ r \end{bmatrix}; \hspace{0.5em} A = \begin{bmatrix} -0.9967 & 0 & 0.6176 \\ 0 & -0.5057 & 0 \\ -0.0939 & 0 & -0.2127 \end{bmatrix}; \\
B &= \begin{bmatrix} 0 & -4.2423 & 4.2423 & 1.4871 \\ 1.6532 & -1.2735 & -1.2735 & 0.0024 \\ 0 & -0.2805 & 0.2805 & -0.8823 \end{bmatrix}; \\
f_w(x) &= \frac{1}{2}\left[\sin(p)\cos^2(p), -\sin(2q), 1\right]^{\T}.
\end{split}
\end{align}
The states consist of roll, pitch and yaw rates while the inputs consist of deflections of the canard, right and left elevons and the rudder. The linear components of these dynamics were established in \cite{SDRA}, and the wind effects consist of sinusoidal and constant expressions as modeled in \cite{WWS24}. As discussed in \cite{PO25b}, these effects are such that the Lipschitz constant of $f_w(x)$ is no more than $1$, and thus $D_f = \|A\|_{\infty}+1 = 2.6143$. Further, $D_g = 0$ as the dynamics are linear in the input, and thus $g_0 = B$ and Assumption \ref{asm:ctrb} is satisfied.

\begin{figure}[!t]
\centering
\begin{subfigure}{0.38\textwidth}
	\centering
	\includegraphics[width = \textwidth]{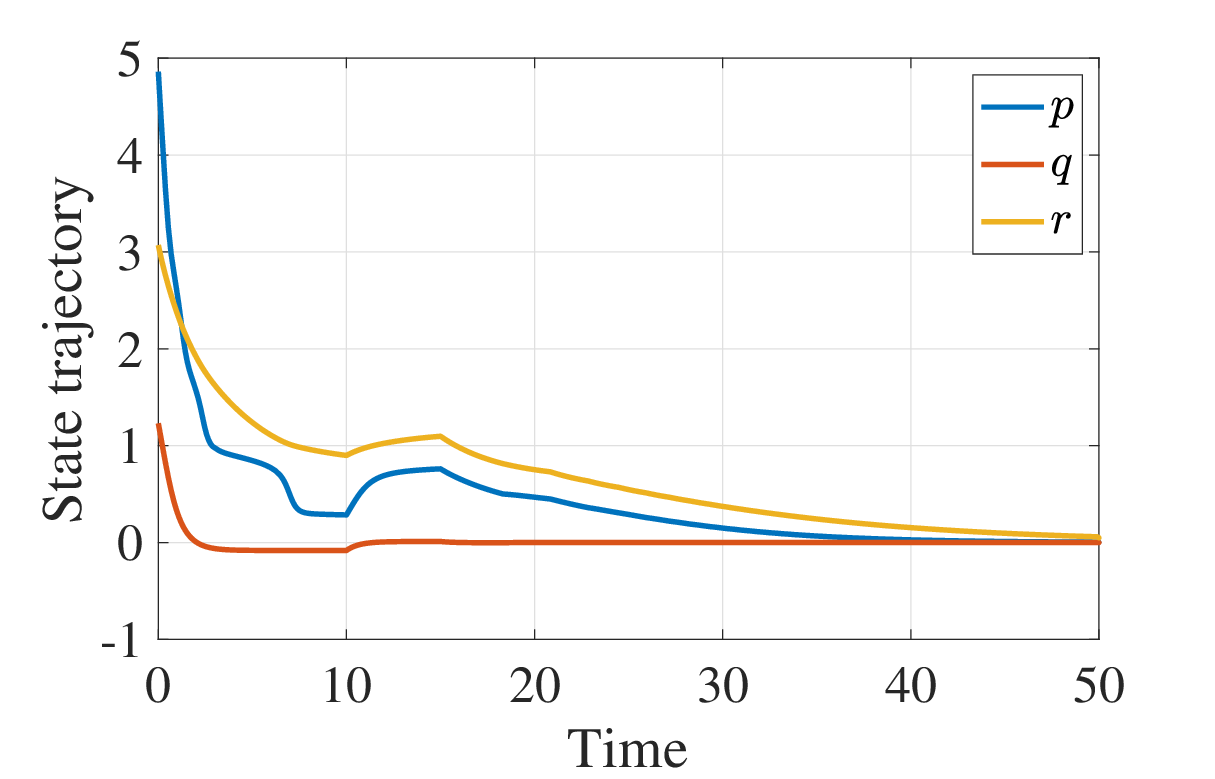}
	\caption{}
\end{subfigure}
\\
\begin{subfigure}{0.38\textwidth}
	\centering
	\includegraphics[width = \textwidth]{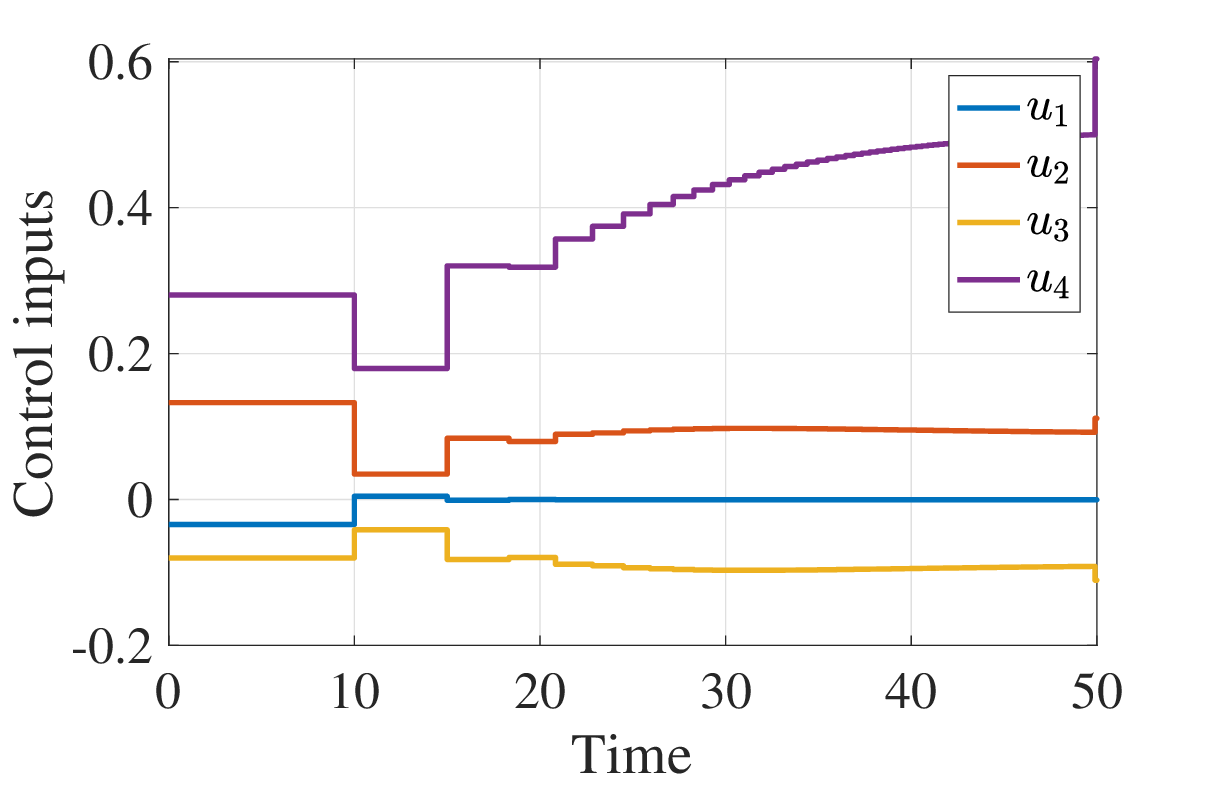}
	\caption{}
\end{subfigure}

\caption{Control of the ADMIRE fighter jet model subjected to wind effects. (a) State trajectories, (b) Control inputs. Clearly, the target state is achieved while satisfying the input constraint.}
\label{fig:ADMIRE}

\end{figure}

We randomly select the initial state $x_0 = [4.86~1.23~3.07]^{\T}$ and fix the target state $x_{tg} = [0~0~0]^{\T}$ with $t^* = 10$. Fig.~\ref{fig:ADMIRE} illustrates state trajectories and control inputs under the sequence of inputs \eqref{eq:un} for this example. It is clearly seen that the target state is achieved while simultaneously satisfying the input constraint \eqref{eq:SetU}. In particular, the piecewise constant nature of the sequence of inputs \eqref{eq:un} is distinctly visible. We remark here that sufficient condition \eqref{eq:t_cond} may not necessarily be satisfied in this example with the chosen value of $t^*$. Nevertheless, it is easily seen that the input constraint remains satisfied.

\begin{figure}[!t]
\centering
\begin{subfigure}{0.38\textwidth}
	\centering
	\includegraphics[width = \textwidth]{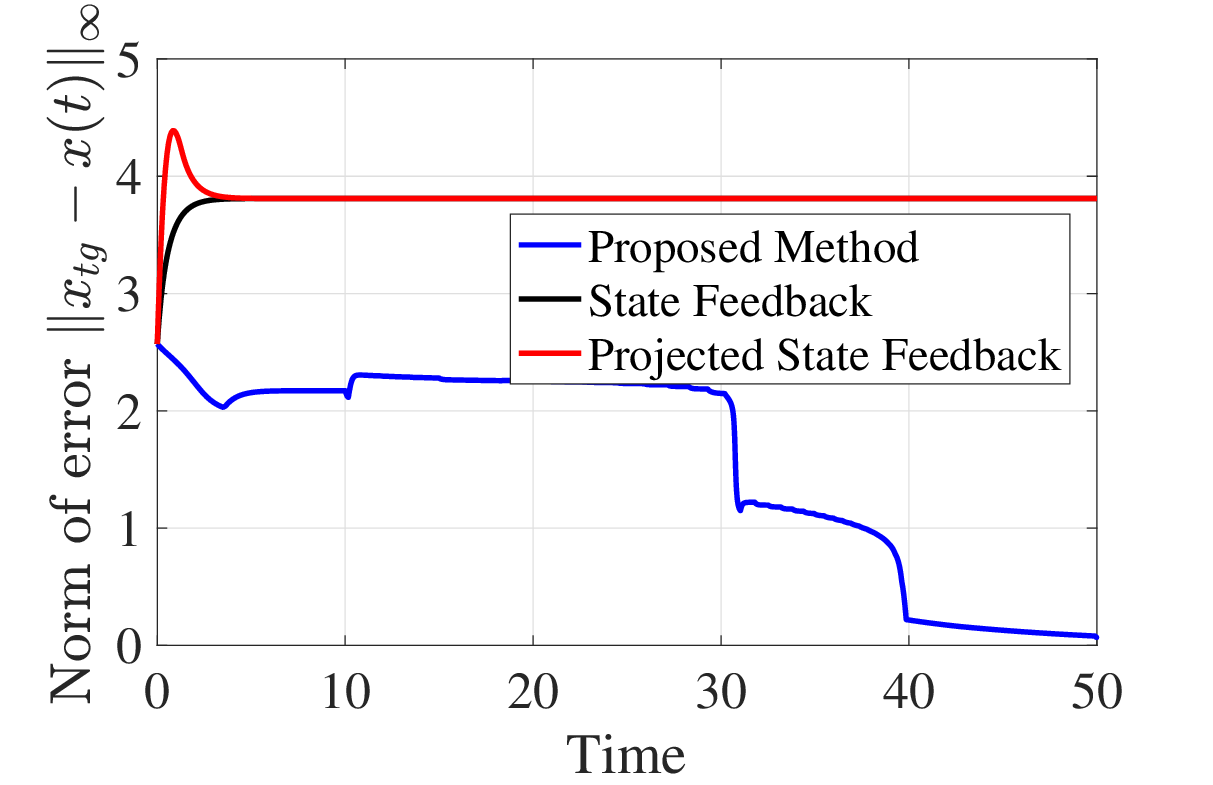}
	\caption{}
\end{subfigure}
\\
\begin{subfigure}{0.38\textwidth}
	\centering
	\includegraphics[width = \textwidth]{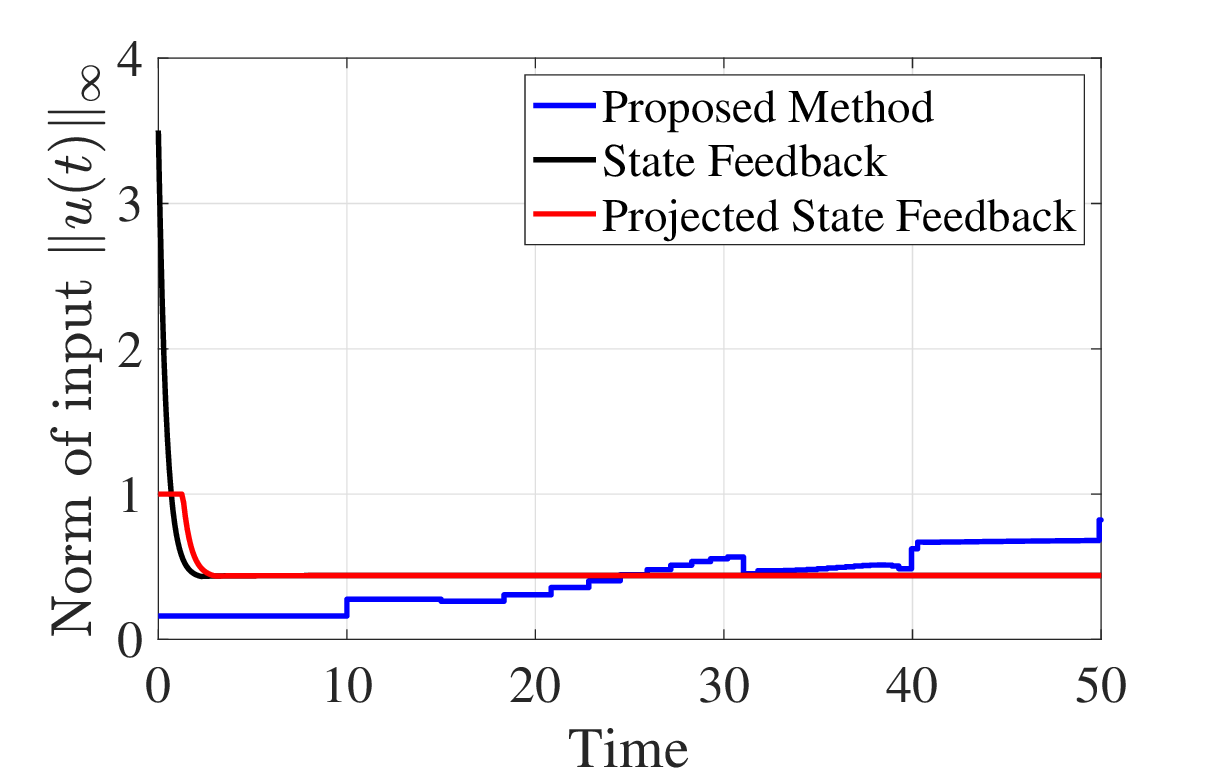}
	\caption{}
\end{subfigure}

\caption{{Comparing the performance of the proposed method with linear state feedback in achieving a non-zero target state. (a) Norm of error, (b) Norm of input. The proposed method outperforms the state feedback controller and the same controller projected onto the constraint set \eqref{eq:SetU}, both in achieving the target state and in satisfying the input constraint.}}
\label{fig:NZ_SF}

\end{figure}

{In Fig.~\ref{fig:NZ_SF}, we compare the performance of the proposed method with $t^* = 10$, to the performance of linear state feedback after linearizing the given dynamics, a procedure described in \cite{Khalil_NS, OC76}. In this case, we randomly select a non-zero target state $x_{tg} = [3.81~2.13~3.73]^{\T}$, starting from a randomly chosen initial state $x_0 = [1.24~2.76~2.08]^{\T}$. We note that nonlinear effects in the model disrupt the performance of linear state feedback, and the error does not reduce to zero, unlike when there are only linear dynamics present. The state feedback controller does not satisfy the input constraint either. Further, on projecting the state feedback controller onto the constraint set \eqref{eq:SetU}, transient performance degrades and the error is still nonzero. In comparison to these, the proposed method ensures that the error reduces to zero while still satisfying the input constraint.}

Finally, in Fig.~\ref{fig:NZ}, we also compare the performance of the proposed method with $t^* = 10$, to the performance of CasADi \cite{Casadi}, an open-source tool for efficient nonlinear model predictive control, interfaced with IPOPT \cite{Ipopt}. In this comparison we randomly select a non-zero target state $x_{tg} = [3.81~2.13~3.73]^{\top}$, starting from a randomly chosen initial state $x_0 = [1.24~2.76~2.08]^{\top}$. We note that both methods achieve the target state as the respective errors converge to zero. However, nonlinear MPC requires significantly more control effort, with at least one of the inputs applying maximum possible effort for an extended period of time. We also compare the computational time of both methods in Table~\ref{tab:Comparison}, where the proposed method clearly outperforms the nonlinear MPC solver. In particular, the computational time is improved by a factor of nearly $5$. Such behavior is in line with our expectations, as nonlinear MPC relies on repeatedly solving a potentially challenging optimization problem.

\begin{figure}[!t]
\centering
\begin{subfigure}{0.33\textwidth}
	\centering
	\includegraphics[width = \textwidth]{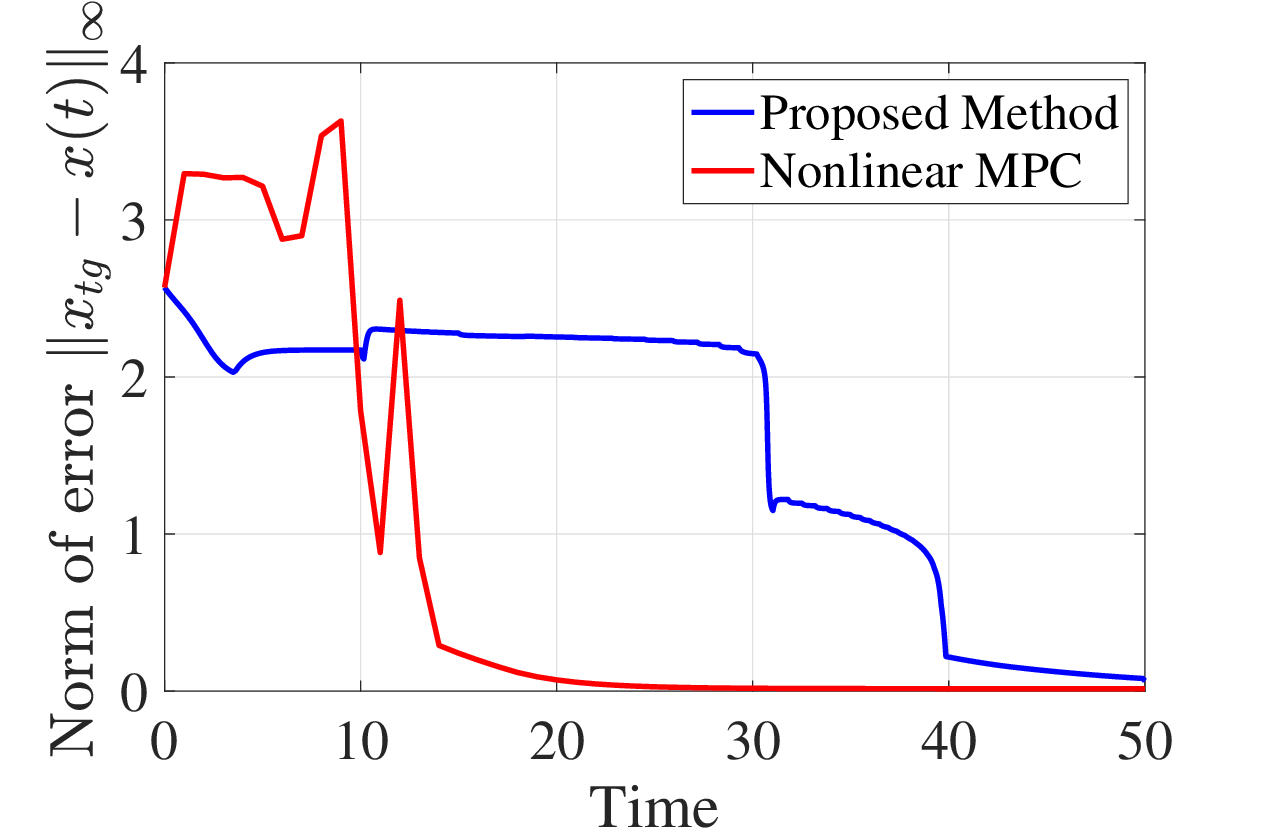}
	\caption{Norm of error}
\end{subfigure}
\\
\begin{subfigure}{0.33\textwidth}
	\centering
	\includegraphics[width = \textwidth]{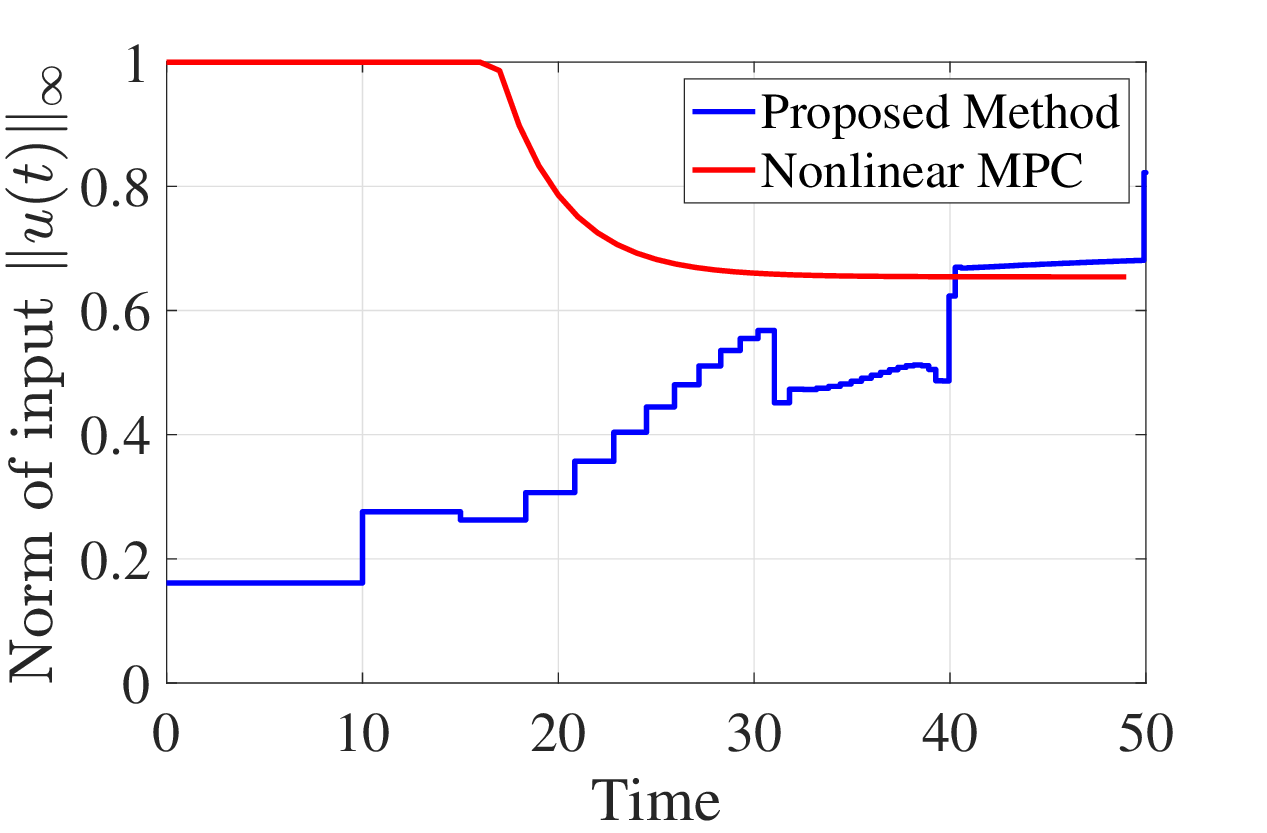}
	\caption{Norm of input}
\end{subfigure}

\caption{{Comparing the performance of the proposed method with nonlinear model predictive control in achieving a non-zero target state. While both methods achieve the target state, nonlinear MPC requires significantly more control effort and takes significantly longer computational time.}}
\label{fig:NZ}

\end{figure}

\begin{table}[!t]
    \centering
    \caption{{Comparing computational time between the proposed method and nonlinear MPC to reach a nonzero target state.}}
    
    {\begin{tabular}{l|cc} \toprule
    \textbf{Method} & Proposed Method & Nonlinear MPC \\ \midrule
    \textbf{Time} & 0.1163s & 0.5774s \\ \bottomrule
    \end{tabular}}
    
    \label{tab:Comparison}
\end{table}

\subsection{Example 2: Forced van der Pol oscillator}
We now consider the dynamics of the forced van der Pol oscillator given by
\begin{align} \label{eq:VDP}
\begin{split}
	\dot{x} &= y, \\
	\dot{y} &= \mu(1-x^2)y - x + u,
\end{split}
\end{align}
where $x$ denotes the position coordinate and $y$ denotes the velocity coordinate. The scalar parameter $\mu$ is an indication of the strength of the nonlinearity. Stabilizing these dynamics using the input $u$ is an important objective in electrical circuits and networks \cite{EN1, EN2, EN3}. 

\begin{figure}[!t]
\centering
\begin{subfigure}{0.38\textwidth}
	\centering
	\includegraphics[width = \textwidth]{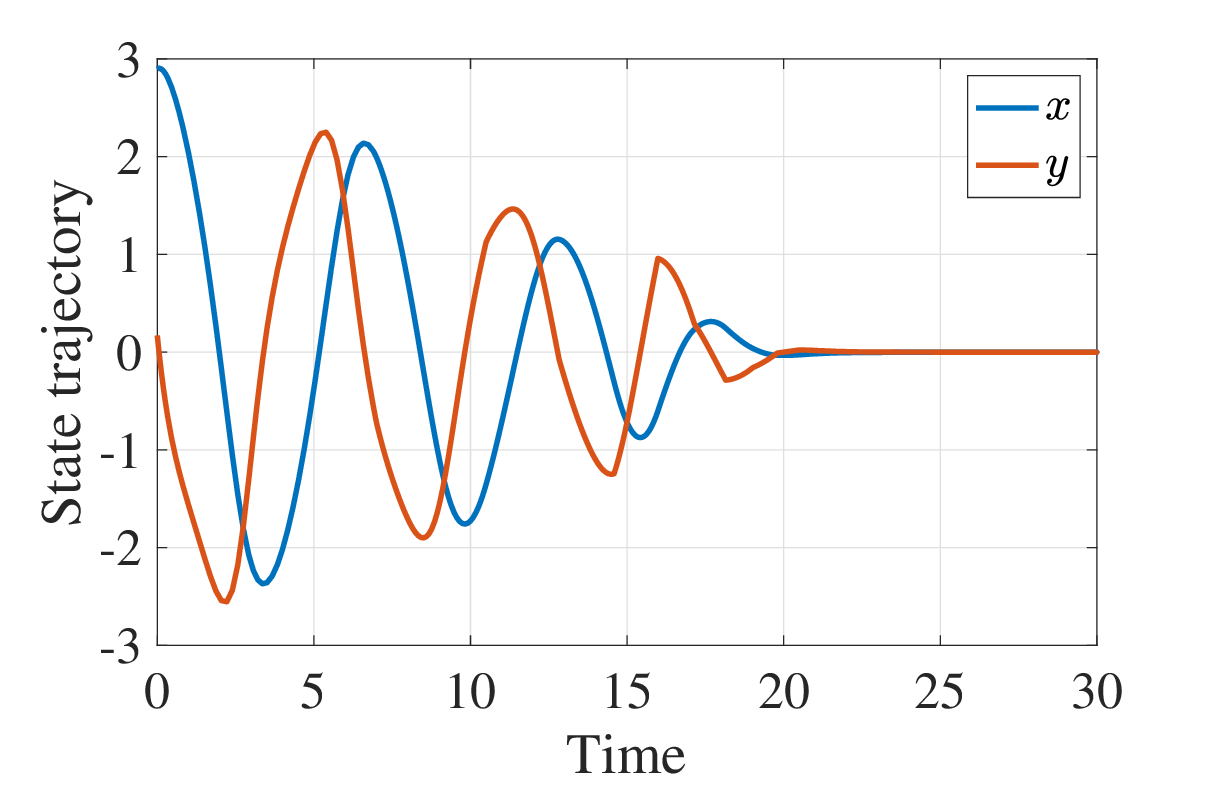}
	\caption{}
\end{subfigure}
\\
\begin{subfigure}{0.38\textwidth}
	\centering
	\includegraphics[width = \textwidth]{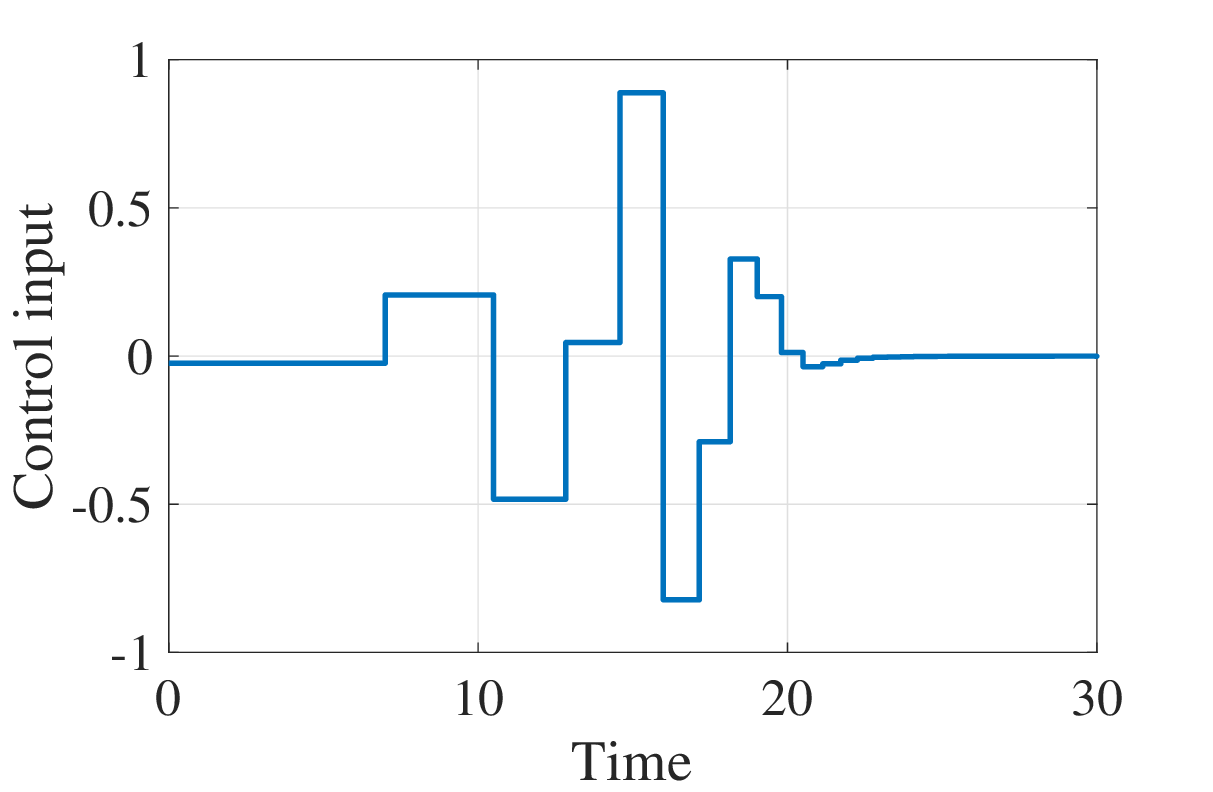}
	\caption{}
\end{subfigure}

\caption{Stabilizing the forced Van der Pol oscillator. (a) State trajectories, (b) Control input. Despite the assumptions in Section \ref{sec:Formulation} not being met, the target state is still achieved while satisfying the input constraint.}
\label{fig:VDP}

\end{figure}

We note that finding $D_f$ and $D_g$ in \eqref{eq:f_cond} and \eqref{eq:g_cond} is not straightforward for any non-trivial choice of $\X$, as the dynamics \eqref{eq:VDP} are not globally Lipschitz. Further, though the dynamics are linear in the input, we note that $g_0 = \begin{bmatrix} 0 \\ 1 \end{bmatrix}$ does not have full row rank, thus violating Assumption \ref{asm:ctrb}. Despite these limitations, this example illustrates that the procedure described in Section~\ref{sec:Method} can be used to stabilize the dynamics \eqref{eq:VDP}. From a randomly chosen initial state $x_0 = [2.90~ 0.17]^{\T}$, we seek to achieve the target state $x_{tg} = [0~ 0]^{\T}$ and set $t^* = 7$.

Fig.~\ref{fig:VDP} illustrates state trajectories and control inputs for this example with $\mu = 0.2$. Once again, it is clear that the target state is achieved while satisfying the input constraint, despite $g_0$ not having full row rank and the assumption of global Lipschitz continuity being violated. Further, $t^*$ is simply chosen as a design parameter, and may not satisfy condition \eqref{eq:t_cond}. In view of Remark \ref{rem:relax}, the additional error arising from violating Assumption \ref{asm:ctrb} decays to zero. However, we show in our next example that there may exist some non-zero steady-state error.

\begin{figure}[!t]
\centering
\begin{subfigure}{0.38\textwidth}
	\centering
	\includegraphics[width = \textwidth]{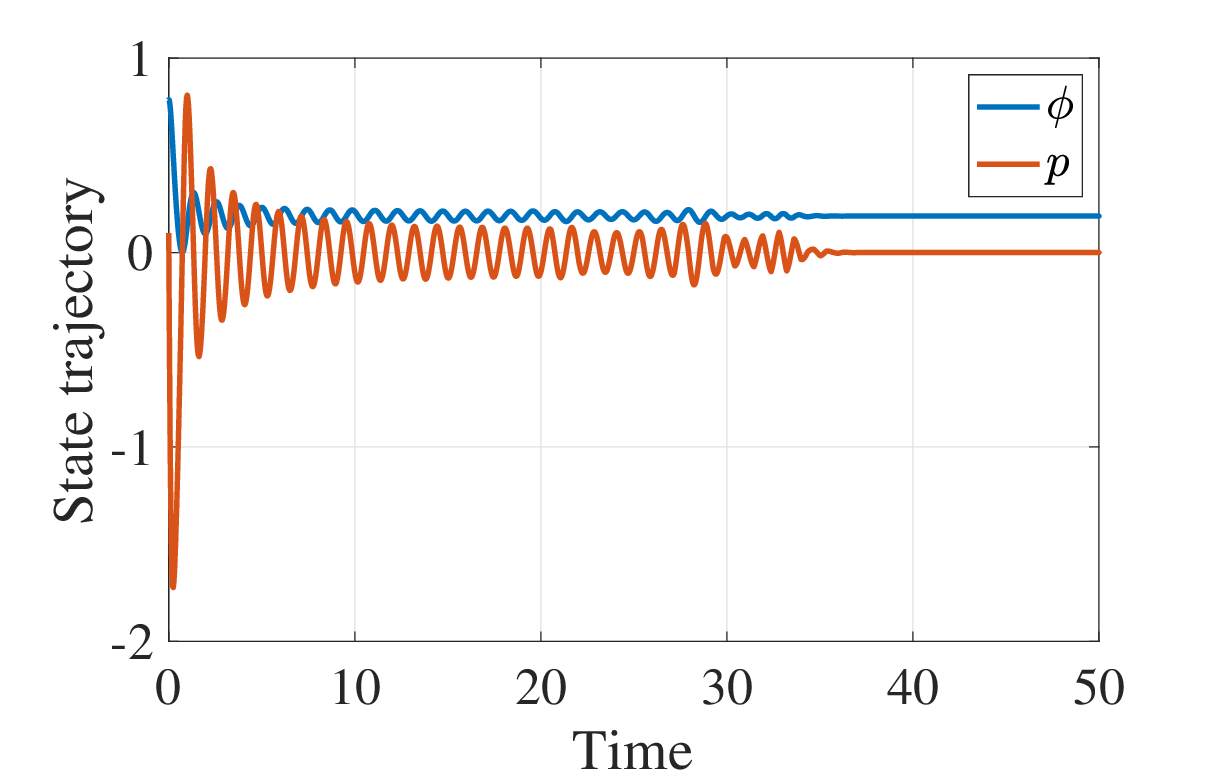}
	\caption{}
\end{subfigure}
\\
\begin{subfigure}{0.38\textwidth}
	\centering
	\includegraphics[width = \textwidth]{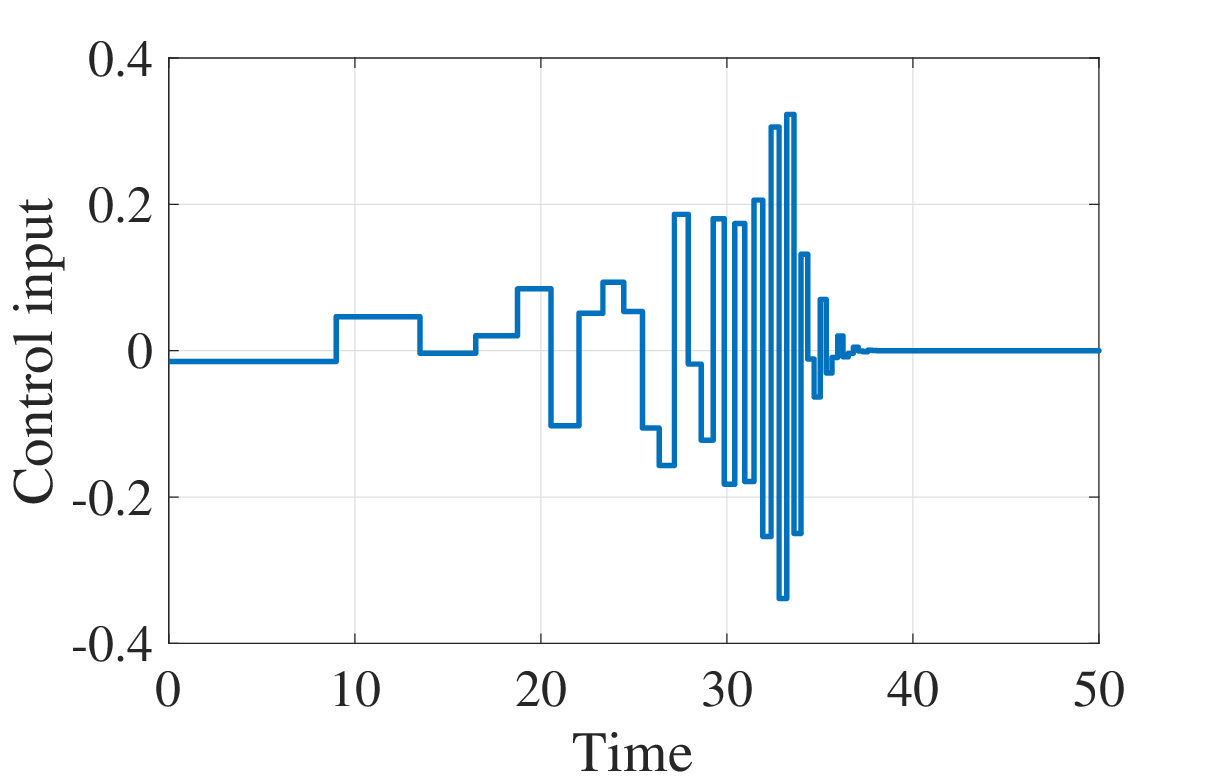}
	\caption{}
\end{subfigure}

\caption{Control of wing rock motion. (a) State trajectories, (b) Control input. The sequence of control inputs \eqref{eq:un} results in bounded error from the target state.}
\label{fig:WR}

\end{figure}

\subsection{Example 3: Wing rock motion}
In this example, we consider the problem of wing rock motion arising in high-performance, delta-winged aircraft at large angles of attack. This phenomenon results in self-induced oscillations of the roll angle due to asymmetric aerodynamic coefficients on the delta wing. The dynamics of this phenomenon can be written based on the formulation in \cite{MK96} as
\begin{align} \label{eq:WR}
\begin{split}
	\dot{\phi} &= p, \\
	\dot{p} &= c_1 + c_2\phi + c_3p + c_4|\phi|p + c_5|p|p + c_6u,
\end{split}
\end{align}
where $\phi$ is the roll angle, $p$ is the roll rate, $u$ is the aileron deflection, and
\begin{align*}
	&c_1 = 5; ~ c_2 = -26.67 \text{ s}^{-1}; ~ c_3 = 0.765 \text{ s}^{-1};\\
	&c_4 = -2.92 \text{ rad/s}; ~ c_5 = -2.5; ~ c_6 = 0.75
\end{align*}
are the aerodynamic coefficients when the angle of attack is $\pi/6 \text{ rad}$ \cite{MK96}. As with the previous example, finding $D_f$ and $D_g$ is not straightforward for non-trivial $\X$. Further, we have $g_0 = \begin{bmatrix} 0 \\ c_6 \end{bmatrix}$, which does not have full row rank. Thus, Assumption \ref{asm:ctrb} is not satisfied. Nevertheless, in line with the insights of Remark \ref{rem:relax}, we show that the procedure of Section \ref{sec:Method} results only in a finite, small error from the target state.

We set $x_0 = [\pi/4~ 0.1]^{\T}$, $x_{tg} = [0~ 0]^{\T}$ and $t^* = 9$. The state trajectories and control inputs in this example are shown in Fig.~\ref{fig:WR}. In contrast to the previous example, we note that the target state is not achieved as the roll angle $\phi$ settles to a non-zero value while the roll rate subsides to zero. This error may be quantified by following steps in Remark \ref{rem:relax}. As before, the chosen value of $t^*$ does not satisfy the second part of condition \eqref{eq:t_cond}, but the input constraint is still satisfied. We also note that the control inputs show significant oscillations from around $t = 20$. While our theory does not explain this behavior, these oscillations subside after around $t = 35$ and the states converge to their final values.

\begin{figure}[!t]
\centering
\begin{subfigure}{0.38\textwidth}
	\centering
	\includegraphics[width = \textwidth]{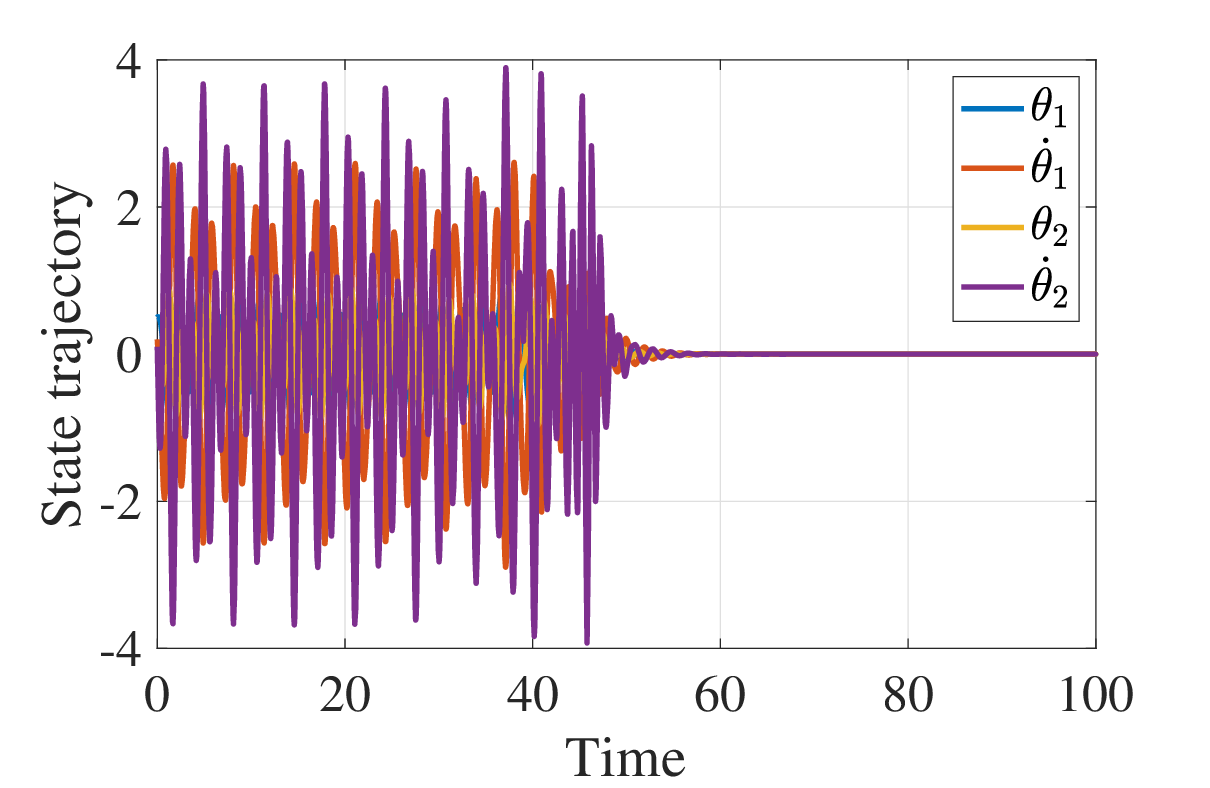}
	\caption{}
\end{subfigure}
\\
\begin{subfigure}{0.38\textwidth}
	\centering
	\includegraphics[width = \textwidth]{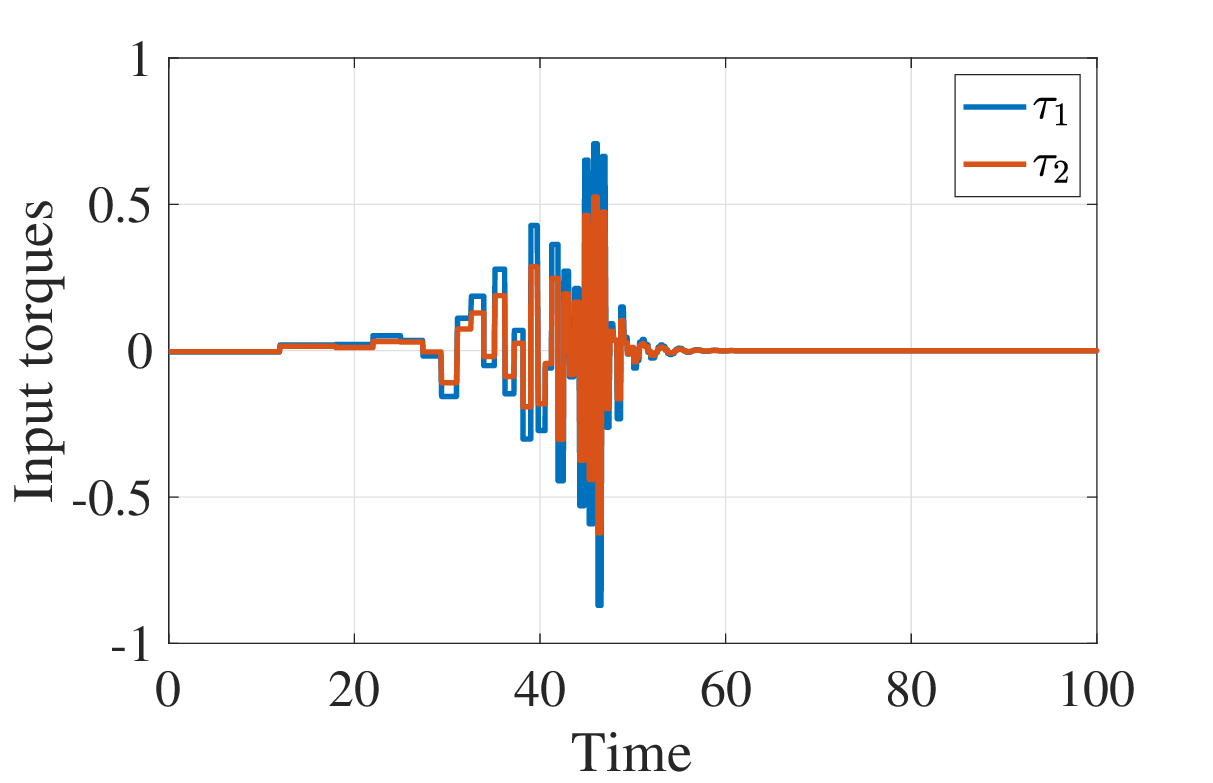}
	\caption{}
\end{subfigure}

\caption{Stabilizing the link positions and velocities of a two-link robot manipulator. (a) State trajectories, (b) Input torques. The target state is achieved despite the significant nonlinearities, though the performance in this example is highly sensitive to system parameters.}
\label{fig:RM}

\end{figure}

\subsection{Example 4: Two-link robot manipulator}
In our final example, we consider the following dynamics of a two-link robot manipulator, derived using the Euler-Lagrange formalism \cite{Nijmeijer}:
\begin{equation} \label{eq:RM}
	\Ddot{\theta} = -M(\theta)^{-1}C(\theta, \dot{\theta}) - M(\theta)^{-1}K(\theta) + M(\theta)^{-1}\tau,
\end{equation}
where
\begin{align*}
M(\theta) &= \begin{bmatrix} m_1l_{1}^{2} + m_2\left|l_1 + l_2\right|^2 & m_2\left(l_{2}^{2} + l_1l_2\cos\theta_2\right) \\ m_2\left(l_{2}^{2} + l_1l_2\cos\theta_2\right) & m_2l_{2}^{2} \end{bmatrix}, \\
C(\theta, \dot{\theta}) &= \begin{bmatrix} -m_2l_1l_2\left(\sin\theta_2\right)\dot{\theta}_2\left(2\dot{\theta}_1+\dot{\theta}_2\right) \\ m_2l_1l_2\left(\sin\theta_2\right)\dot{\theta}_{1}^{2} \end{bmatrix}, \\
K(\theta) &= \begin{bmatrix} \left(m_1+m_2\right)gl_1\sin\theta_1 + m_2gl_2\sin\left(\theta_1 + \theta_2\right) \\ m_2gl_2\sin\left(\theta_1 + \theta_2\right) \end{bmatrix}. 
\end{align*}
The vector $\theta = [\theta_1~ \theta_2]^{\T}$ consists of coordinates of positions of the two links and $\tau = [\tau_1~ \tau_2]^{\T}$ is the vector of input torques on both links. The respective lengths of the links are $l_1$ and $l_2$, their respective masses are $m_1$ and $m_2$ and $g$ is the acceleration due to gravity. In the inertia matrix $M(\theta)$, $\left|l_1+l_2\right|^2 = l_{1}^{2} + l_{2}^{2} + 2l_1l_2\cos\theta_2$. 

The state vector $x = [\theta_1~ \dot{\theta_1}~ \theta_2~ \dot{\theta}_2]^{\T}$, and since there are only two input torques, Assumption \ref{asm:ctrb} is violated again. Further, due to the structure of the matrices $M(\theta)$ and $C(\theta, \dot{\theta})$, finding $D_f$ and $D_g$ is not straightforward. We set $m_1 = 0.3, m_2 = 0.4, l_1 = 0.2, l_2 = 0.5$, and aim to stabilize the dynamics \eqref{eq:RM} to the origin from $x_0 = [0.5~ 0.2~ 0~ 0.1]^{\T}$ using $t^* = 12$.

The state trajectories and control inputs in this example are shown in Fig.~\ref{fig:RM}. Despite significant oscillations, the states eventually stabilize to the origin and the input constraint is satisfied. We note that this example is particularly sensitive to the choice of parameters, and slight changes in the initial condition or system parameters may cause the input constraint to be violated. As with the previous example, there are also significant oscillations in the input signal, though these eventually stabilize along with the system states.

\section{Discussion} \label{sec:Discussion}
In this paper, we presented an approach to control constrained Lipschitz nonlinear systems through a sequence of controllers designed for their linear driftless approximation at the initial state. We first showed that optimal controllers to achieve a target state for driftless dynamics lead to bounded error from the target in the nonlinear system. By applying these controllers over successively shorter time intervals, we showed that this error monotonically decreases and can be made arbitrarily small by considering sufficiently many intervals. We also derived conditions that ensure actuation constraints are satisfied. A set of case studies demonstrate the efficacy of this approach, with satisfactory performance in achieving a target state even when underlying assumptions are violated. In this section, we provide some additional remarks, discussing merits and drawbacks of this technique and present a number of avenues for future research.

We note that in comparison to methods such as feedback linearization or Jacobian linearization, our approach does not require complete knowledge of system dynamics. Instead, we use only the value of the input function $g$ at the initial state, and thus the controller uses a very simple approximation of the original system dynamics. This technique also allows developing conditions which ensure the input constraint is satisfied, while this is not guaranteed in linearization-based approaches. State constraints that limit how far $x(t)$ may be from $x_{tg}$ may be addressed by appropriately designing $t^*$, since $x_1$ is within $\vbar_1$ of $x_{tg}$, and each subsequent $x_n$ is even closer to $x_{tg}$ using Lemma~\ref{lem:vn}. Further, applying the sequence of controllers \eqref{eq:un} is less computationally intensive than more advanced techniques such as nonlinear model predictive control or neural network controllers.

Despite these advantages, our theoretical guarantees require assuming that the nonlinear dynamics are globally Lipschitz and the driftless dynamics are fully controllable. As we discuss in Remark \ref{rem:relax}, relaxing the latter assumption can still lead to satisfactory performance, though potentially resulting in some steady-state error. Relaxing the globally Lipschitz assumption leads to more difficulties, since each of the bounds $\vbar_n$ in \eqref{eq:vn_bound} themselves depend on the Lipschitz constants $D_f$ and $D_g$. Thus, restricting dynamics to be locally Lipschitz requires a careful characterization of the domain, and such an endeavor is outside the current scope of the paper. Nevertheless, the examples in Section \ref{sec:Examples} show that our approach can still result in satisfactory tracking performance when these assumptions are broken.

It is interesting to examine whether other driftless approximations will provide guarantees similar to \eqref{eq:vn_bound} and \eqref{eq:vn_rec}. In particular, the matrix $g_0$ is only one possible driftless approximation, and it is perhaps more natural to use $g(x_{tg})$ or $g(x_{eq})$ instead to construct \eqref{eq:Driftless}, where $x_{eq}$ is an equilibrium state of \eqref{eq:System}, assuming one exists. In this context, there exist a number of avenues to extend the techniques in this paper, and we discuss some of these below.

\begin{itemize}[label = $\bullet$, leftmargin = 1.2em]

\item \emph{Finite-horizon and resilient control:} While this paper focused on achieving a target state over an infinite horizon, the sequence of time instants $\{t_n\}$ can be suitably modified to achieve a target state in a prescribed finite time $t_f$. For instance, we may choose $t_n = t_f\left(\frac{a^n-1}{a^n}\right)$ for some $a > 1$ such that $\lim_{n \to \infty} t_n = t_f$ and $\lim_{n \to \infty} \Delta t_n = 0$. This sequence converges geometrically to $t_f$, and hence the corresponding sequence of inputs can be shown to have bounded energy, which is not true for \eqref{eq:un}. Such a property can be exploited to design resilient nonlinear systems in the context of \cite{BO22b, PO25b}, which we will describe in detail in a forthcoming paper.

\item \emph{Systems subject to bounded disturbances:} In the same spirit as relaxing Assumption \ref{asm:ctrb}, we note that nonlinear systems subject to bounded disturbances or noise might also be stabilized to a neighborhood of the target state using the sequence of inputs \eqref{eq:un}, with bounded steady-state error. The bound in \eqref{eq:vn_bound} may then be directly affected by the bound on the disturbance.

\item \emph{Online learning and control synthesis:} The work of Meng \emph{et al.} \cite{MSWO24} presents an online method to drive a nonlinear system to a target state. This is accomplished through the use of an underapproximated proxy system and synthesis of a control law by learning from a single trajectory run. In this context, we are currently developing methods that use the synthesis procedure in Section \ref{sec:Method} in combination with learning from a single trajectory run to develop new online controllers for Lipschitz nonlinear systems. It is worth noting that Farsi \emph{et al.} \cite{FLYL22} present a similar piecewise learning approach to control nonlinear systems.

\item \emph{Nonlinear consensus problems:} In a network of agents indexed by $i = 1, \ldots, n_A$, the Linear Consensus Protocol (LCP) \cite{ME10} can be formulated based on a set of communicating integrators with dynamics $\dot{x}_i = u_i$ where $u_i$ is designed to achieve \emph{consensus} among agents' values. These dynamics can be viewed as a linear driftless system that approximates the dynamics of a network of communicating agents with individual nonlinear dynamics of the form \eqref{eq:System}. On this basis, similar to the procedure in Section \ref{sec:Method}, we aim to investigate whether using the LCP over successively shorter time intervals can lead to consensus in a nonlinear multi-agent network. A key difficulty to address is whether an input constraint of the form \eqref{eq:SetU} is applicable in this context.
	

\end{itemize}


\balance
\bibliographystyle{plain}        
\bibliography{references}           

\end{document}